\documentclass[letter,11pt]{article}
\usepackage[margin=1in]{geometry}
\usepackage[utf8]{inputenc} 
\usepackage[T1]{fontenc}
\usepackage{amsmath,amssymb,subcaption,graphicx,amsthm,enumitem,mathpazo,xcolor,tikz-cd, multicol,comment,lipsum,leftindex} 
\usepackage{hyperref}
\usepackage[capitalize,nameinlink,noabbrev,nosort]{cleveref}

\graphicspath{{ImagesEulerian/}}

\title{A combinatorial proof of an identity involving Eulerian numbers}

\author{Jer\'onimo Valencia Porras}
\date{}

\theoremstyle{plain}
\newtheorem{theorem}{Theorem}[section]
\newtheorem{prop}[theorem]{Proposition}
\newtheorem{lemma}[theorem]{Lemma}

\newtheorem{conjecture}[theorem]{Conjecture}
\theoremstyle{definition}
\newtheorem{definition}[theorem]{Definition}

\newtheorem{example}[theorem]{Example}
\newtheorem{observation}[theorem]{Observation}
\newtheorem{question}[theorem]{Question}
\newtheorem{remark}[theorem]{Remark}
\crefname{defn}{Definition}{Definitions}
\crefname{theorem}{Theorem}{Theorems}
\crefname{lemma}{Lemma}{Lemmas}

\renewcommand{\to}{\longrightarrow}

\newcommand{\tcb}[1]{\textcolor{blue}{#1}}

\DeclareMathOperator{\vol}{vol}
\DeclareMathOperator{\alc}{alc}
\DeclareMathOperator{\lin}{lin}
\DeclareMathOperator{\aff}{aff}
\DeclareMathOperator{\word}{word}
\DeclareMathOperator{\words}{words}
\DeclareMathOperator{\comp}{comp}
\DeclareMathOperator{\conv}{conv}
\DeclareMathOperator{\pair}{pair}
\DeclareMathOperator{\des}{des}

\begin{document}

\maketitle
\begin{abstract}
    We give a combinatorial proof of an identity that involves Eulerian numbers and was obtained algebraically by Brenti and Welker (2009). To do so, we study alcoved triangulations of dilated hypersimplices. As a byproduct, we describe the dual graph of the triangulation in the case of the dilated standard simplex and the hypersimplex, conjecture its structure for dilated hypersimplices. Also, we give a new proof of the classical result that the Eulerian numbers coincide with the normalized volumes of the hypersimplices. 
\end{abstract}

%\tableofcontents

\section{Introduction}

Brenti and Welker \cite{BRENTI2009545} studied the transformation of the numerator of rational formal power series after taking a subsequence of the coefficients and computing its generating series, motivated by the Veronese construction for graded algebras. In particular, they show how the coefficients of the numerator transform under this operation. 
We summarize their results and, to do so, we provide relevant definitions. A \emph{weak composition} of a nonnegative number $n$ is an ordered sequence of nonnegative numbers $\vec{c} = (c_1,c_2,\ldots,c_\ell)$ such that $c_1+c_2+\ldots+c_\ell = n$. 

\begin{definition}\label{def:composition-set}
    For $d,r,i\in\mathbb{N}$ and $d\geq 1$ let $$\mathfrak{C}(r,d,i) := \left\{ (c_1,c_2,\ldots,c_d)\in\mathbb{N}^d \; \big | \; c_1+c_2+\ldots+c_d = i \; , \; c_j\leq r \text{ for } 1\leq j\leq d \right\}.$$ Denote by $C(r,d,i)$ the size of the set $\mathfrak{C}(r,d,i)$.
\end{definition}

A \emph{partition} $\lambda$ of a nonnegative number $n$ is a sequence of nonnegative numbers in weakly decreasing order $\lambda_1\geq \lambda_2 \geq\ldots\geq \lambda_\ell \geq 0$ that sum to $n$. The numbers $\lambda_i$ are the \emph{parts of $\lambda$} and for a given integer number $k\geq 0$, the \emph{multiplicity of $k$ in $\lambda$}, denoted by $m_k(\lambda)$, is the number of times $k$ appears in the sequence. For $d\geq 1$, the value of $C(r,d,i)$ can be computed as $$C(r,d,i) = \sum_{\{\lambda \subseteq (r^d) \, : \, |\lambda| = i\}} \binom{d}{m_1(\lambda),m_2(\lambda),\ldots,m_r(\lambda),d-\ell(\lambda)}$$ where the sum runs over all partitions $\lambda$ of $i$ with biggest part $\lambda_1 \leq r$ and at most $d$ parts. The Veronese construction for formal power series is the following transformation, as described in \cite[Theorem 1.1]{BRENTI2009545}. Suppose $f(z)$ is a formal power series with complex coefficients satisfying $$f(z) = \sum_{n\geq 0} a_nz^n = \frac{h(z)}{(1-z)^d}$$ for some polynomial $h(z) = h_0+h_1z+\ldots+h_sz^s$ and with $d,s \geq 0$. Then for any positive integer $r\geq 1$,  $$f^{\langle r \rangle}(z) := \sum_{n\geq 0} a_{rn}z^n = \frac{h^{\langle r \rangle}_0+h^{\langle r \rangle}_1z+\ldots+h^{\langle r \rangle}_mz^m}{(1-z)^d}$$ where $m = \max(s,d)$ and, for $i=1,2,\ldots,m$, $$h^{\langle r \rangle}_i = \sum_{j=0}^s C(r-1,d,ir-j)h_j.$$

The transformation $h(z) \mapsto h^{\langle r \rangle}(z)$ has been studied recently \cite{Beck2010,RealRootednessJochemko2018,SymmetricJochemko2022}, and it has an interesting interpretation in the case of the Ehrhart theory of lattice polytopes. If $P$ is a lattice polytope, taking $f(z)$ to be the \emph{Ehrhart series of $P$}, $h^{\langle r \rangle}(z)$ is the $h^*$-polynomial of the dilated polytope $rP$. In particular, the result previously mentioned yields the $h^*$-vector of $rP$ as a linear combination of the corresponding vector of the original polytope (see \cite{beck2007computing} for the relevant definitions). 

Going further, the authors describe the Veronese construction using two different bases for the ring of formal power series of the form specified before; the first is the "monomial basis", and the second is related to the \emph{Eulerian polynomials}. By considering a change of basis for this transformation, they showed the following identity involving the numbers $C(r,d,i)$ and \emph{the coefficients of the Eulerian polynomials} (see \cref{def:Eulerian-numbers} and \cref{rem:eulerian-convention} for a discussion on these numbers). 

\begin{prop}\cite[Prop. 2.3]{BRENTI2009545}\label{prop:identities}
    Let $d,r \geq 1$. Then 
    \begin{equation}\label{eq:identity-general-i}
        \sum_{j=0}^d C(r-1,d+1,ir-j)A(d,j) = r^d A(d,i)
    \end{equation}
    for $i=0,1,\ldots,d.$ In particular, when $i=1$,
    \begin{equation}\label{eq:identity-i-1}
        \sum_{j=0}^d C(r-1,d+1,r-j)A(d,j) = r^d.
    \end{equation}
\end{prop}

We point out that the derivation of these two equations is completely algebraic, and the appearance of both $C(r,d,i)$ and $A(d,j)$ is a consequence of algebraic considerations. Given the enumerative nature of the numbers in the equations, Brenti and Welker asked for a combinatorial proof of these identities. Our main result is to provide one such proof by constructing two pairs of bijections that yield these equations as corollaries. The key idea that allows us connect both sides of the equations is to consider dilations of hypersimplices, in particular their \emph{alcoved triangulations} (see \cref{sec:alcoved,sec:hypersimplices} for the definitions of the objects in the theorem). 

\begin{theorem}\label{thm:principal}
    Let $\mathcal{A}(r\Delta_{i,d})$ be the set of alcoves of the $r$-dilated hypersimplex $\Delta_{i,d}$. There exist bijections 
    \begin{align*}
        \words_i \; &:\; \mathcal{A}(r\Delta_{i,d})\longrightarrow [r]^d \times \mathfrak{A}(d,i) \\
        \pair_i\; &: \;\mathcal{A}(r\Delta_{i,d})\longrightarrow \bigcup_{j=1}^d \mathfrak{C}(r-1,d+1,ir-j)\times \mathfrak{A}(d,j)
    \end{align*}
    from which we obtain a combinatorial proof of \cref{eq:identity-general-i}.
\end{theorem}

This document is organized as follows. In \cref{sec:background} we review Eulerian polynomials and a combinatorial identity for their coefficients, and the relevant results about alcoved polytopes. In \cref{sec:delta-1-d-simplices-labels} we construct the bijections needed to show \cref{eq:identity-i-1} by considering dilated standard hypersimplices; we also study the dual graph of the alcoved triangulation of these polytopes. We generalize and use these ideas in \cref{sec:delta-i-d-simplices-labels} to build the bijections of \cref{thm:principal}; moreover, we describe the dual graphs of the alcoved triangulations of general hypersimplices and conjecture the structure of these graph for dilated hypersimplices. 

\begin{remark}
    After completion of this manuscript, we found \cite[Proposition 9.4]{alcoved2} where the authors give a formula for the volume of \emph{thick hypersimplices} and specializes to \cref{eq:identity-general-i} when $\Phi = A_n$ and the parameters are carefully chosen. We point out that their argument relies on considerations of the alcoved triangulations related to affine Weyl groups and, in contrast, we provide a proof of \cref{eq:identity-general-i,eq:identity-i-1} by interpreting alcoves in terms of combinatorial objects. 
\end{remark}

\section{Preliminaries}\label{sec:background}

In this section we present the objects that are involved in the combinatorial proof that we aim to present. We point out that the indexing and some notation of these objects may differ from the literature. As a starting point, we denote by $[r]^d$ the set of strings of length $d$ in the elements $[r] := \{1,2,\ldots,r\}$. We refer to such a string as a \emph{word}, and the elements constituting the word are its \emph{letters.}

\subsection{Eulerian numbers and Eulerian polynomials}\label{sec:Eulerian_nums_polys}

For an in-depth introduction to these Eulerian object we refer the reader to \cite{petersen2015eulerian}. We summarize the pertinent facts about them in this section. Originally, Euler considered the power series  $$\sum_{k\geq 0} (k+1)^d z^k = \frac{A_d(z)}{(1-z)^{d+1}}$$ while studying the Riemann $\zeta$-function \cite{euler-riemman}. The numerators $A_d(z)$ are known as the \emph{Eulerian polynomials}\footnote{These are different from the \emph{Euler} polynomials $E_n(z)$ defined through the exponential generating series $$\sum_{n\geq 0}E_n(z)\frac{t^n}{n!} = \frac{2e^{zt}}{e^t +1}.$$}. If we write $A_d(z) = \sum_{j=0}^{d-1}  a_{d,j}\,z^j$, by taking derivatives we obtain that the coefficients of these polynomials satisfy $$a_{d+1,k} = (k+1)\,a_{d,k} + (n+1-k)\,a_{d,k-1}.$$ 

This linear recurrence characterizes a particular combinatorial object. The \emph{permutations of $d$ elements} is the set $\mathfrak{S}_d$ of all bijections from $[d]:=\{1,2,\ldots,d\}$ to itself. An element $\sigma \in \mathfrak{S}_d$ can be represented in \emph{one-line notation} by listing its values in order. That is, the one-line notation of $\sigma$ is the word $\sigma(1)\,\sigma(2)\,\ldots\,\sigma(d)$, sometimes written as $\sigma_1\,\sigma_2\,\ldots,\sigma_d$. A \emph{descent} of a permutation $\sigma\in\mathfrak{S}_d$ is a number $i\in[d]$ such that $\sigma(i) > \sigma(i+1)$. This corresponds to "going down" when reading the one-line notation from left to right. The number of descents of $\sigma$ is denoted by $\des(\sigma).$ By taking a permutation on $d$ elements and $k-1$ or $k$ descents and counting the different possibilities to turn it into a permutation on $d+1$ elements and $k$ descents, the linear recurrence mentioned before follows. In other words, $a_{d,k}$ counts the permutations of $d$ elements with $k$ descents. These numbers are the \emph{Eulerian numbers}. For our purposes, and to make notation line-up in \cref{sec:hypersimplices}, we need to shift the indices of these numbers. Hence, we adopt the following convention. 

\begin{definition}\label{def:Eulerian-numbers}
    Let $d\geq 1$ and $1\leq j \leq d$. Define $$\mathfrak{A}(d,j) := \left\{ \sigma\in \mathfrak{S}_d \; \big | \; \des(\sigma) = j-1 \right\}.$$ We denote the cardinality of this sets by $A(d,j) = |\mathfrak{A}(d,j)|$.
\end{definition}

\begin{remark}\label{rem:eulerian-convention}
    With the previous definition, the Eulerian numbers are given by $a_{d,k} = A(d,k+1)$. We will refer to the numbers $A(d,k)$ as the Eulerian numbers for exposition reasons, but there is a necessary word of caution when interpreting them combinatorially.\\
\end{remark}

\subsection{Alcoved polytopes}\label{sec:alcoved}

Given a finite set of vectors in $\mathbb{R}^n$, their \emph{convex hull} is the smallest convex set containing all of them. A \emph{polytope} is the convex hull of finitely many vectors in $\mathbb{R}^n$. The \emph{vertices} of $P$ is the smallest set of vectors such that their convex hull equals $P$. Dually, polytopes can be described as an intersection of finitely many hyperplanes. The \emph{linear span} of $P$ is the (affine) vector space generated by the vectors in the polytope and it is denoted by $\lin(P).$ A polytope $P$ has \emph{dimension $d$} if $\lin(P)$ is a $d$-dimensional (affine) space, or in other words, if $\lin(P) \cong \mathbb{R}^{d}$. Under this isomorphism, the lattice $\mathbb{Z}^{d} \subset \mathbb{R}^{d}$ corresponds to a lattice inside of $\lin(P)$, which is referred to as the \emph{affine span of $P$} and is denoted by $\aff(P).$ 
A \emph{lattice polytope} is a polytope whose vertices only have integer coordinates. We say that two lattice polytopes $P$ and $Q$ are \emph{affinely equivalent} if there exists an affine map that restricts to a bijection from $P$ to $Q$ and also an isomorphism from $\aff(P)$ to $\aff(Q)$. Note that this definition does not require the polytopes to have the same dimension. 
For a complete introduction on polytopes we refer the reader to \cite{ziegler2012lectures}. 

We focus on a particular family of lattice polytopes known as \emph{alcoved polytopes}. In the rest of this section we review the relevant information about them from \cite{lam2007alcoved}. 
There are two ways in which alcoved polytopes usually appear in the literature, depending on the coordinates that are chosen to describe them. To deal with this, we define $\mathbb{R}_{x}^n$ to be the Euclidean space with points with coordinates $(x_1,\ldots,x_n)$. 

\begin{definition}\label{def:alcoved_polytopes}
    An \emph{alcoved polytope} $P$ is a polytope that has one of the following hyperplane descriptions:
    \begin{itemize}
        \item An $(H,z)$-representation $$P = P(b_{ij},c_{ij}) = \left\{ (z_1,z_2,\ldots,z_{n-1})\in\mathbb{R}_z^{n-1} \; : \; b_{ij} \leq z_i-z_j \leq c_{ij} \quad \text{for} \quad 0\leq i < j \leq n-1 \right\}$$ with $z_0 := 0$ and $b_{ij},c_{ij} \in \mathbb{Z}$ for all $i$ and $j$.

        \item An $(H,x)$-representation $$P = P(b_{ij},c_{ij},k) = \left\{ (x_1,x_2,\ldots,x_n)\in\mathbb{R}_x^{n} \; : \; 
        \begin{matrix}
            b_{ij} \leq x_{i+1}+\cdots+x_j \leq c_{ij} \;\; \text{for} \;\; 0\leq i < j \leq n \\ 
            x_1+x_2+\cdots+x_n = k
        \end{matrix}
        \right\}$$ where $k\in \mathbb{Z}$ and $b_{ij},c_{ij} \in \mathbb{Z}$ for all $i$ and $j$.
    \end{itemize}
\end{definition}

\begin{remark}\label{rem:equivalent_defs}
    Both of these descriptions give rise to affinely equivalent polytopes. Indeed,  consider the maps $\varphi \, : \, \mathbb{R}_x^{n} \to \mathbb{R}_z^{n-1}$ given by $$\varphi(x_1,x_2,\ldots,x_n) = (x_1\,,\,x_1+x_2\,,\,x_1+x_2+x_3\,,\ldots,\,x_1+x_2+\cdots+x_{n-1})$$ and $\psi_k\, : \, \mathbb{R}_z^{n-1} \to \mathbb{R}_x^{n}$ given by $$\psi_k(z_1,z_2,\ldots,z_{n-1}) = (z_1\,,\,z_2-z_1\,,\ldots,\,z_{n-1}-z_{n-2}\,,\,k-z_{n-1}).$$ 
\end{remark}

The most important alcoved polytopes for our purposes are the hypersimplices. 

\begin{definition}\label{def:hypersimplices}
    The \emph{$i$-th hypersimplex of dimension $d$}, denoted by $\Delta_{i,d}$, is the polytope with $(H,x)$-representation
    \begin{equation}\label{eq:hypersimplex}
        \Delta_{i,d} = \left\{ (x_1,x_2,\ldots,x_d)\in\mathbb{R}_x^d \; :\; \begin{matrix}
                0 \leq x_{i}\leq 1 \;\; \text{for} \;\; 1\leq i \leq d \\ 
                x_1+x_2+\cdots+x_d = i
        \end{matrix}
        \right\}.
    \end{equation}
    
    The \emph{standard simplex of dimension $d$} is the first hypersimplex of dimension $d$, that is $\Delta_{1,d}$. It can also be described as the convex hull of the standard basis vectors in $\mathbb{R}_x^n$. A \emph{unimodular simplex} is a polytope $S$ that is affinely equivalent to $\Delta_{1,d}$.
\end{definition}

\begin{remark}
    With the notation we use we are emphasizing the dimension of the hypersimplex rather than the dimension of the space in which its $(H,x)$-version lies in. Notice that this conflicts with the notation of \cite{lam2007alcoved}.
\end{remark}

\begin{example}
    Under the affine map defined in \cref{rem:equivalent_defs}, the standard simplex of dimension $d$ is mapped to $$\varphi(\Delta_{1,d}) = \conv\left\{\vec{1} - \sum_{j=1}^{i-1} \vec{e}_j\in \mathbb{R}_z^{n-1} \; : \;i=1,2,\ldots,n\right\}$$ where $\vec{1}$ is the all-ones vector and $\left\{\vec{e}_j\,:\,j=1,2,\ldots,n-1\right\}$ is the standard basis of $\mathbb{R}_z^{n-1}$. 
\end{example}

A \emph{subdivision} of a polytope $P$ is a collection of polytopes $\mathfrak{P}$ such that every face of a polytope in $\mathfrak{P}$ is also in $\mathfrak{P}$, any two polytopes in $\mathfrak{P}$ intersect in a common face and the union of all the polytopes in $\mathfrak{P}$ equals $P$. If all the polytopes in the collection are (unimodular) simplices, we call the subdivision a \emph{(unimodular) triangulation}. One of the many interesting features of alcoved polytopes is that they come equipped with a unimodular triangulation to which we refer as the \emph{alcoved triangulation}. It is induced by the \emph{affine Coxeter arrangement of type $A_{n-1}$} that subdivides $\mathbb{R}_z^{n-1}$ into unimodular simplices called \emph{alcoves} (see \cref{fig:alcoved_example} for the case $n=3$). 

\begin{example}\label{ex:alcoved_polytopes}

The alcoved polytope $P = P(b_{ij},c_{ij})$ in $z$-coordinates with parameters $b_{0,1} = -4$, $c_{0,1} = -1$, $b_{0,2} = -3$, $c_{0,2} = -1$, $b_{1,2} = -2$, $c_{1,2} = 1$ is depicted in \cref{fig:alcoved_example}. The map $\psi_2(z_1,z_2) = (z_1,z_2-z_1,2-z_2)$ is an affine equivalence to a polytope in $x$-coordinates. After translating such polytope by the vector $(1,1,1),$ we obtain a polytope that lays on the hyperplane $x_1+x_2+x_3 = 5$ in $\mathbb{R}^3_x.$

\begin{figure}
    \centering
    \includegraphics[width=0.95\linewidth]{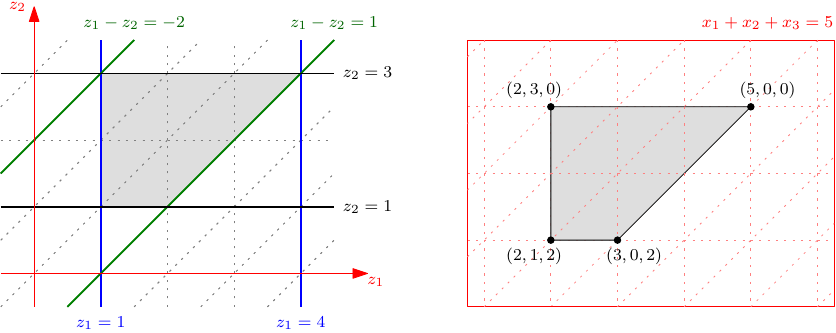}
    \caption{On the left, an alcoved polytope with its explicit $(H,z)$-representation; the dotted lines represent the elements of the affine Coxeter arrangement of type $A_2$. On the right, the image of the polytope under the map $\psi_2$ from \cref{rem:equivalent_defs} after translation by the vector $(1,1,1)$.}
    \label{fig:alcoved_example}
\end{figure}    
\end{example}

For any subdivision, the maximal polytopes with respect to inclusion fully determine the whole subdivision by taking finite intersections. Hence, we make reference to alcoved triangulations by only considering their full-dimensional simplices.

\begin{definition}
    For an alcoved polytope $P$, let $\mathcal{A}(P)$ be the set of maximal simplices (with respect to inclusion) in the alcoved triangulation of $P$. An element $A \in\mathcal{A}(P)$ is an \emph{alcove of $P$.}
\end{definition}

We usually identify the elements of $\mathcal{A}(P)$ with the set of their vertices. With this perspective, Lam and Postnikov \cite{lam2007alcoved} gave a combinatorial description of the alcoves, which we now present. 

\begin{definition}\label{def:sorted-multisets}
    Let $\mathcal{I} = \{I_1,I_2,\ldots,I_k\}$ be a collection of $r$-multisets of $\{1,2,\ldots,n\}$ where for each multiset we assume $I_j = \{I_{j1}\leq I_{j2} \leq \ldots \leq I_{jr}\}$. We say that the collection $\mathcal{I}$ is \emph{sorted} if $$I_{11}\leq I_{21} \leq \ldots I_{k1} \leq I_{21} \leq I_{22} \leq \ldots \leq I_{kr}.$$ 
    Denote by $M_{\mathcal{I}}$ the \emph{matrix associated to the collection of multisets} $\mathcal{I}$  constructed by using the (ordered) multisets as rows. Hence, $\mathcal{I}$ is sorted if the concatenation of columns of $M_{\mathcal{I}}$ from left to right and top to bottom is weakly increasing. 
\end{definition}

\begin{definition}\label{def:multisets_from_vector}
    For an nonnegative integer vector $\vec{a} \in \mathbb{N}^n$ such that $a_1+a_2+\cdots+a_n = r$, let $I_{\vec{a}}$ be the $r$-multiset of $\{1,2,\ldots,n\}$ with $a_i$ elements ``$i$'' for each $i$. For a collection of vectors $A = \{\vec{a}_1,\vec{a}_2,\ldots,\vec{a}_k\} \subseteq \mathbb{N}^n$ such that the coordinates of all of them sum to $r$, define the \emph{collection of multisets of $A$} as $\mathcal{I}_A = \{I_{\vec{a}_1},I_{\vec{a}_2},\ldots,I_{\vec{a}_k}\}.$
\end{definition}

Suppose $P$ has a $(H,x)$-representation in $\mathbb{R}_x^{n}$ such that all points of $P$ have nonnegative coordinates. If this is not the case, by translating $P$ using the vector $m\vec{1} = (m,m,\ldots,m)\in\mathbb{R}_x^n$ for a sufficiently large $m\in\mathbb{Z}$ we obtain an affinely equivalent alcoved polytope with the desired property. Denote by $Z_P = P \cap \mathbb{Z}^{n} \subseteq \mathbb{N}^n$ the set of lattice points of $P$. The following theorem is a reformulation of the characterization of the alcoves of $P$ due to Lam and Postnikov \cite[Theorem 3.1]{lam2007alcoved}.

\begin{theorem}\label{thm:alcoves-sorted-sets}
    Let $P \subseteq \mathbb{R}_x^{n}$ be an alcoved polytope lying in the hyperplane $x_1+x_2+\cdots+x_n = k$ such that all its points have nonnegative coordinates. A simplex with vertices $A = \{\vec{a}_1,\vec{a}_2,\ldots,\vec{a}_n\} \subseteq Z_P$ is an alcove in $\mathcal{A}(P)$ if and only if $\mathcal{I}_A$ is a sorted collection of $k$-multisets. 
\end{theorem}

\begin{example}
    Consider again the polytope in $x$-coordinates from \cref{ex:alcoved_polytopes}. Both of the following claims can be verified in \cref{fig:alcoved_example}. 
    
    The convex hull of the set $A = \{(3,2,0),(4,1,0),(3,1,1)\}$ is a unimodular simplex in the hyperplane $x_1+x_2+x_3 = 5$. For this set of vertices, $\mathcal{I}_A = \{ \{1,1,1,2,2\},\{1,1,1,1,2\},\{1,1,1,2,3\} \}$ is the collection of multisets, where the order of the vertices matches the order of the multisets. This collection is sorted if we pick $I_1 = \{1,1,1,1,2\}$, $I_2 = \{1,1,1,2,2\}$, and $I_3 = \{1,1,1,2,3\}$. In \cref{fig:alcoved_example} it can be seen that the set $A$ does determine an alcove of $P$. 

    Now, the convex hull of the set $B = \{ (2,1,2),(3,0,2),(2,2,1) \}$ also gives a unimodular simplex in $x_1+x_2+x_3 = 5$. In this case, $\mathcal{I}_B = \{ \{1,1,2,3,3\},\{1,1,1,3,3\},\{1,1,2,2,3\} \}$ is the collection of multisets. If $\mathcal{I}_B$ were sorted, then $I_1 = \{1,1,1,3,3\}$ since the first two coordinates of all the sets are equal. Now, $\{1,1,2,2,3\}$ has to be put before $\{1,1,2,3,3\}$ but the it conflicts with the chosen $I_1$. Hence, $\mathcal{I}_B$ is not sorted, meaning that it does not correspond to an element in $\mathcal{A}(P).$
\end{example}

\section{Alcoved triangulations of dilated hypersimplices}\label{sec:hypersimplices}

In this section we prove \cref{eq:identity-general-i,eq:identity-i-1}. As mentioned before, the main idea of the proof is to understand their right-hand side geometrically. 

For a $n$-dimensional polytope $P\subset \mathbb{R}^n$, its \emph{volume} is defined by the (Riemann) integral $\text{Vol}(P) = \int_P d\vec{x}$. This amounts to assigning volume $1$ to the standard cube $[0,1]^n$. If we assign volume $1$ to the standard simplex instead, we obtain the \emph{normalized volume} of the polytope, denoted by $\vol(P)$. If $P$ is not $n$-dimensional but still lies in $\mathbb{R}^n$, the volume computations can be performed relative to the linear span of $P$ in order to obtain non-zero volume for objects such as alcoved polytopes with $(H,x)$-representation (see \cite[Section 5.4]{beck2007computing} for the general discussion on \emph{relative volume} and \cite[Theorem 3.2]{lam2007alcoved} for the particular case of alcoved polytopes). In what follows, we omit the adjective "relative" as it plays no significant role for our discussion.  

\begin{remark}
    In view of the previous comments, the normalized volume of an alcoved polytope $P$ is given by $|\mathcal{A}(P)|$.
\end{remark} 

Thus, in order to compute the volume of an alcoved polytope it is enough to understand the set $\mathcal{A}(P)$. We include the following notation to ease the reading of the rest of the section.

\begin{definition}
    Let $P$ be an alcoved polytope. A \emph{labeling} of $\mathcal{A}(P)$ using the elements of a finite set $S$ is a bijection $f\,:\,\mathcal{A}(P) \to S$.
\end{definition}

We now return to the hypersimplices. A famous result from Laplace \cite{laplace1886oeuvres} states that the normalized volume of $\Delta_{i,d}$ is $A(d,i)$. The first triangulation of the hypersimplex that showed this identity combinatorially was constructed by Stanley \cite{Stanley-triang-hypersimplex}. We give another combinatorial proof of this result in \cref{sec:delta-i-d-simplices-labels} by constructing a labeling of $\mathcal{A}(\Delta_{i,d})$ with permutations in $\mathfrak{S}_d$ with $i-1$ descents. Finally, the right-hand side of \cref{eq:identity-general-i} can be rewritten as $$r^d A(d,i) = r^d \vol(\Delta_{i,d}) = \vol(r\Delta_{i,d})$$ where $r\Delta_{i,d}$ is the dilation of the hypersimplex by a factor of $r$. 

In the rest of the section we first show \cref{eq:identity-i-1} by constructing two different labelings of $\mathcal{A}(r\Delta_{1,d})$ and then use those ideas to construct the corresponding labelings for $\mathcal{A}(r\Delta_{i,d})$ that show \cref{eq:identity-general-i}.

\subsection{The dilated standard simplex}\label{sec:delta-1-d-simplices-labels}

We want to describe a labeling of the alcoves of $$r\Delta_{1,d} = \left\{ \vec{x}\in\mathbb{R}_x^{d+1} \;:\; 0\leq x_1\,,\,x_2\,,\,\ldots\,,\,x_{d+1} \leq r \quad \text{and} \quad x_1+x_2+\cdots+x_{d+1} = r \right\}$$ using words in $[r]^d$ and pairs in $\bigcup_{j=1}^d \mathfrak{C}(r-1,d+1,r-j)\times \mathfrak{A}(d,j)$. From the description of the polytope it is clear that that the lattice points of $r\Delta_{1,d}$ are the weak compositions of $r$ with $d+1$ parts. 

\subsubsection{Labeling of the alcoves with words}

We reinterpret the sorted sets $\mathcal{I}_A$ from \cref{thm:alcoves-sorted-sets} using words. 

\begin{definition}
    Let $\mathcal{I} = \{I_1,I_2,\ldots,I_{k}\}$ be a sorted collection of different $r$-multisets of $\{1,2,\ldots,n\}.$ The \emph{decorated matrix} $\widetilde M_\mathcal{I}$ is constructed as follows. Arrange the numbers of $M_\mathcal{I}$ in a $k\times r$ grid and then (using matrix coordinates)
    \begin{itemize}
        \item[(a)] if $I_{ab}<I_{(a+1)b}$ mark the edge between the cells $(a,b)$ and $(a+1,b)$ in the grid, and
        \item[(b)] if $I_{nb}<I_{1(b+1)}$ mark the bottom edge of cell $(n,b)$ in the grid. 
    \end{itemize}
\end{definition}

\begin{example}\label{ex:decorated-matrix}
    Fix $n=8$ and $r = 6$. Consider the set of points 
    \begin{multline*}
    A = \{\vec{a}_1,\vec{a}_2,\vec{a}_3,\vec{a}_4,\vec{a}_5\} = \{(2,1,0,1,1,0,0,1),(2,0,1,1,1,0,0,1),\\(1,1,0,2,0,1,0,1),(1,1,0,1,1,1,0,1),(1,1,0,1,1,0,1,1)\}\subseteq \mathbb{R}^8
    \end{multline*}
    The decorated matrix $\widetilde M_{\mathcal{I}}$ for $\mathcal{I} = \mathcal{I}_A = \{I_{\vec{a}_1},I_{\vec{a}_2},I_{\vec{a}_3},I_{\vec{a}_4},I_{\vec{a}_5}\}$ is 
    
    \begin{center}
        \begin{tikzpicture}[scale = 0.7]
            \def \w{1};
            \def \h{1};
            \def \r{0.3};
            \def \rows{5};
            \def \cols{6};
            \foreach \i in {0,...,\rows}
            {
            \draw[gray!50] (0,\i*\h)--(\w*\cols,\i*\h);
            }
            \foreach \i in {0,...,\cols}
            {
            \draw[gray!50] (\w*\i,0)--(\w*\i,\rows*\h);
            }
            \foreach \xx\yy\c in {0/4/1,0/3/1,0/2/1,0/1/1,0/0/1,
                                  1/4/1,1/3/1,1/2/2,1/1/2,1/0/2,
                                  2/4/2,2/3/3,2/2/4,2/1/4,2/0/4,
                                  3/4/4,3/3/4,3/2/4,3/1/5,3/0/5,
                                  4/4/5,4/3/5,4/2/6,4/1/6,4/0/7,
                                  5/4/8,5/3/8,5/2/8,5/1/8,5/0/8
                                  }
            {
            \node at (\w*.5+\w*\xx,\h*.5+\h*\yy) {\c};
            }
            \foreach \xx\yy in {1/3,2/4,2/3,3/2,4/3,4/1,4/0}
            {
            \draw[red,ultra thick] (\xx,\yy) --  (\xx+1,\yy);
            }
            %\node at (2.5,-5.5) {$(N_0,Q_0)$};
        \end{tikzpicture}
    \end{center}
\end{example}

We now examine the case of sorted sets arising from the alcoves of the dilated standard simplex. The following lemma gives conditions on the decorated matrices that appear in this case. 

\begin{lemma}\label{lem:properties-marked-matrix-alcoves-dilated-simplex}
    Let $A = \{\vec{a}_1,\vec{a}_2,\ldots,\vec{a}_{d+1}\}$ be the set of vertices of an alcove of $r\Delta_{1,d}$ and let $\mathcal{I} = \mathcal{I}_A$ be the associated sorted collection of $r$-multisets. Then the decorated matrix $\widetilde M_{\mathcal{I}}$ has the following properties:
    \begin{enumerate}
        \item For each $1\leq i\leq d$, there is a unique mark between rows $i$ and $i+1$, and 
        \item there are no marks in the bottom part of the matrix. 
    \end{enumerate}
\end{lemma}

\begin{proof}
    First, note that $\vec{a}_i \neq \vec{a}_j$ if $i\neq j$ since $A$ is the set of vertices of a simplex. Hence, the set $\mathcal{I}_A$ consists of $d+1$ different $r$-multisets of the set $\{1,2,\ldots,d+1\}$. 
    Since all of the multisets are distinct, there is at least one mark between each pair of adjacent rows in $\mathcal{M}_{\mathcal{I}_A}$. Moreover, since the maximum element of each multiset is at most $d+1$, the decorated matrix has at most $d$ marks. Otherwise, we would obtain an entry in the matrix that is greater than $d+1.$ Combining these two facts we obtain the lemma.  
\end{proof}

\cref{lem:properties-marked-matrix-alcoves-dilated-simplex} allows us to define a labeling of $\mathcal{A}(r\Delta_{1,d})$ by reading the positions of the marks in each of the rows. 

\begin{definition}
    Let $\word_1 : \mathcal{A}(r\Delta_{1,d}) \to [r]^d\;$ be the map defined as follows. If $A$ is the set of vertices of an alcove of $r\Delta_{1,d}$ with associated collection of multisets $\mathcal{I} = \mathcal{I}_A$ then $\word_1(A)$ is obtained by reading the column label of the marks of $\widetilde M_\mathcal{I}$ from top to bottom. 
\end{definition}

\begin{example}\label{ex:word-simplex-delta-d}
    We give an example where $d = 4$ and $r = 6$. The set of points 
    \begin{equation*}
    A = \{(3,1,1,0,1),(2,2,1,0,1),(2,2,0,1,1),(2,1,1,1,1),(2,1,1,0,2)\}
    \end{equation*}
    defines an alcove of $6\Delta_{1,4}\subseteq \mathbb{R}^5_x$. The decorated matrix in this case is 
    \begin{center}
        \begin{tikzpicture}[scale = 0.7]
            \def \w{1};
            \def \h{1};
            \def \r{0.3};
            \def \rows{5};
            \def \cols{6};
            \foreach \i in {0,...,\rows}
            {
            \draw[gray!50] (0,\i*\h)--(\w*\cols,\i*\h);
            }
            \foreach \i in {0,...,\cols}
            {
            \draw[gray!50] (\w*\i,0)--(\w*\i,\rows*\h);
            }
            \foreach \xx\yy\c in {0/4/1,0/3/1,0/2/1,0/1/1,0/0/1,
                                  1/4/1,1/3/1,1/2/1,1/1/1,1/0/1,
                                  2/4/1,2/3/2,2/2/2,2/1/2,2/0/2,
                                  3/4/2,3/3/2,3/2/2,3/1/3,3/0/3,
                                  4/4/3,4/3/3,4/2/4,4/1/4,4/0/5,
                                  5/4/5,5/3/5,5/2/5,5/1/5,5/0/5
                                  }
            {
            \node at (\w*.5+\w*\xx,\h*.5+\h*\yy) {\c};
            }
            \foreach \xx\yy in {2/4,3/2,4/3,4/1}
            {
            \draw[red,ultra thick] (\xx,\yy) --  (\xx+1,\yy);
            }
        \end{tikzpicture}
    \end{center}
    and then $\word_1(A) = 3\,5\,4\,5\in [6]^4.$
\end{example}

\begin{theorem}\label{thm:bijection-alcoves-words}
    The map $$\word_1 \; : \; \mathcal{A}(r\Delta_{1,d}) \to [r]^d$$ is a bijection. 
\end{theorem}
\begin{proof}
    \cref{lem:properties-marked-matrix-alcoves-dilated-simplex} shows that $\word_1$ is a well-defined map, that is, the image is in the set $[r]^d$. 
    To show that it is a bijection, we give an inverse. Given a word $w \in [r]^d$, we construct a decorated matrix as follows: start by considering a $(d+1)\times r$ grid with a mark in between rows $i$ and $i+1$ in column $w_i$ for $1\leq i \leq d$. Fill the entries of the matrix by setting a ``$1$'' in position $(1,1)$ and using the following procedure: 
    \begin{itemize}
        \item For $1\leq i \leq d$, if the cell $(i,j)$ has label ``$k$'' and there is no mark right below it, label the cell $(i+1,j)$ with ``$k+1$''. 
        \item If the cell $(d+1,j)$ has label ``$k$'', label the cell $(1,j+1)$ with ``$k$''.
    \end{itemize}
    Say this produces the matrix $M$. If $M_i$ denotes the $i$-th row of $M$, then considered as multisets, $\{M_1,M_2,\ldots,M_{d+1}\}$ is a collection of sorted $r$-multisets of $\{1,2,\ldots,d+1\}$. Let the vector $\vec{a}_i$ be the indicator vector of $M_i$. Since $M_i$ has size $r$, the sum of the entries of $\vec{a}_i$ is $r$. Then the set $A = \{\vec{a}_1,\vec{a}_2,\ldots,\vec{a}_{d+1}\}$ is an alcove of a polytope in $\mathbb{R}_x^{d+1}$ lying in the hyperplane $x_1+x_2+\cdots+x_{d+1} = r$. Since all $\vec{a}_i$ have coordinates adding up to $r$, all these integer points are part of $r\Delta_{1,d}$. Hence, they are the vertices of an alcove of this polytope. 
\end{proof}

\subsubsection{Labeling of the alcoves with pairs of compositions and permutations}\label{sec:labeling_pair_case_1}

We start by associating a composition to each collection of nonnegative integer vectors $V$. The idea behind this composition is that it is the additive inverse of the maximal translation that $\conv(V)$ allows so that the polytope remains in the positive orthant. 

\begin{definition}\label{def:composition-set-of-vectors}
    Let $V=\{\vec{v}_1,\vec{v}_2,\ldots,\vec{v}_{m}\}\subseteq\mathbb{N}^{d+1}$ be a collection of vectors with nonnegative integer coordinates. Define $\comp(V) = (c_1,c_2,\ldots,c_{d+1})$ to be the composition with parts $c_k = \min\left \{(\vec{v}_{j})_{k}\,|\,j\in[m]\right \}$ where $(\vec{v}_j)_k$ denotes the $k$-th entry of the vector $\vec{v}_j.$
\end{definition}

If we restrict to collections of vertices of alcoves of a dilated standard simplex, we can give a description of the associated composition using the decorated matrix of the collection. 

\begin{lemma}\label{lem:comp-A-from-decorated-matrix}
    Let $A=\{\vec{a}_1,\vec{a}_2,\ldots,\vec{a}_{k}\}\in\mathcal{A}(r\Delta_{1,d})$. Then $\comp(A) = (c_1,c_2,\ldots,c_{d+1})$ satisfies the following in terms of the marks of $M_{\mathcal{I}_A}$:
    \begin{enumerate}
        \item $c_1 = k$ if the first mark is in column $k+1$,
        \item Let $2\leq i \leq d$. Suppose the $i$-th mark is in column $b_i$. Then $$c_{i} = \begin{cases}b_{i}-b_{i-1}-1 \quad &\text{if the $i$-th mark is higher than the $(i-1)$-th mark}\\ b_{i}-b_{i-1} \quad & \text{if the $i$-th mark is higher than the $(i-1)$-th mark}\\\end{cases}$$
        \item $c_{d+1} = \ell$ if the last mark is in column $r-\ell$.  
    \end{enumerate}
\end{lemma}

\begin{proof}
    The lemma follows from the fact that $c_j$ is equal to $\left\lfloor \frac{N_j}{d+1} \right\rfloor$ where $N_j$ is the number of entries ``$j$'' in the matrix $M_{\mathcal{I}_A}$. Indeed, if $N_j = k(d+1)+\ell$ for some $k,\ell\in\mathbb{N}$ with $\ell < d+1$, then each $\vec{a}_i$ has at least $k$ coordinates equal to $j$, and some of them have exactly $j$ coordinates equal to $j$. Hence $c_j = k.$
\end{proof}

Now we associate a permutation to each of the alcoves of the dilated hypersimplices. 

\begin{definition}\label{def:permutation-alcove-dilated-simplex}
    Let $A=\{\vec{a}_1,\vec{a}_2,\ldots,\vec{a}_{d+1}\}\in\mathcal{A}(r\Delta_{1,d})$. Define $\sigma_A \in S_{d}$ to be the permutation such that its one-line notation of is the word obtained from $\widetilde M_{\mathcal{I}_A}$ by recording the position of the marks reading the columns from top to bottom and from left to right in the matrix.
\end{definition}    

\begin{example}\label{ex:comp-permutation-simplex-delta-d}
    For the alcove $A\in\mathcal{A}(6\Delta_{1,4})$ from \cref{ex:word-simplex-delta-d}, $\comp(A) = 2\,1\,0\,0\,1$, and $\sigma_A = 1\,3\,2\,4$ in one-line notation. 
\end{example}

\begin{prop}\label{prop:properties-comp-A-and-permutation-A}
    Let $A=\{\vec{a}_1,\vec{a}_2,\ldots,\vec{a}_{d+1}\}\in\mathcal{A}(r\Delta_{1,d})$. Then $\comp(A) \in \mathfrak{C}(r-1,d+1,r-j)$ for some $j\in\{1,2,\ldots,d\}$ and in this case $\sigma_A \in \mathfrak{A}(d,j)$. 
\end{prop}

\begin{proof}
    Let $\mathcal{I} = \mathcal{I}_A$. From the definition of $\comp(A)$ we observe that it has $d+1$ parts and each of them is at most $r-1$. We now give a characterization of the integer $j$ for which $\comp(A) \in \mathfrak{C}(r-1,d+1,r-j)$ in terms of the decorated matrix $\widetilde{M}_\mathcal{I}$. Using the notation from \cref{lem:comp-A-from-decorated-matrix}, and letting $N$ be the number of marks that are lower than the next one, 
    \begin{align*}
        \sum_{i=1}^{d+1} c_i &= k + \sum_{i=1}^{d} (b_{i}-b_{i-1}) - N + \ell \\
        &= (b_1-1) + b_d - b_1 - N + (r - b_d) = r - (N+1).
    \end{align*}
    Therefore, $j= N + 1$. Now, note that each mark that is lower than the following one contributes to a descent of $\sigma_A$ according to the definition of the permutation. Hence $ j = N + 1 = \des(\sigma_A) + 1$ meaning that $\sigma_A \in \mathfrak{A}(d,j)$ as desired. 
\end{proof}

With these two objects we can construct the following labeling. We prove that this is actually a bijection at the start of \cref{sec:delta-i-d-simplices-labels} after discussing alcoves of hypersimplices in detail.

\begin{theorem}\label{thm:bijection-alcoves-pairs}
    The map $$\pair_1\,:\,\mathcal{A}(r\Delta_{1,d})\to \bigcup_{j=1}^d \mathfrak{C}(r-1,d+1,r-j)\times \mathfrak{A}(d,j)$$ given by $\pair_1(A) = (\comp(A),\sigma_A)$ is a bijection.
\end{theorem}

\noindent From \cref{thm:bijection-alcoves-words,thm:bijection-alcoves-pairs}, we obtain a combinatorial proof of \cref{eq:identity-i-1}. 

\subsubsection{Dual graph of the triangulation}

\begin{definition}
    Let $\mathcal{T}$ be a triangulation of a polytope $P$. The \emph{dual graph} of the triangulation $G_\mathcal{T}$ has vertex set equal to the maximal simplices of $\mathcal{T}$ and two such simplices $S_1$ and $S_2$ form an edge, which we denote by $S_1 \sim S_2$, whenever their intersection has \emph{codimension $1$}, that is the dimension of the polytope $S_1\cap S_2$ is $\dim(P)-1$.  
\end{definition}

From \cref{thm:bijection-alcoves-words}, the maximal simplices of the alcoved triangulation of $r\Delta_{1,d}$ are labeled by words in $[r]^d$. The following theorem gives a description of the dual graph of this triangulation in terms of words. 

\begin{definition}\label{def:G_rd}
    For $r,d \geq 1$, let $G_{r,d}$ be the graph on vertex set $[r]^d$ and edges given by
    \begin{enumerate}
        \item $w_1\,w_2\,\ldots\, w_d \; \sim \; (w_d +1)\, w_1\,w_2\,\ldots\, w_{d-1}\;$ whenever $1\leq w_d < r$, and 
        \item $w_1\,\ldots\,w_i\,w_{i+1}\,\ldots\, w_d \; \sim \; w_1\,\ldots\,w_{i+1}\,w_i\,\ldots\, w_d\;$ for any $1\leq i\leq d-1$ such that $w_i\neq w_{i+1}$.
    \end{enumerate}
\end{definition}

\begin{theorem}\label{thm:dual_graph_alcoved_dilated_standard_simplex}
    Let $\mathcal{T}$ be the alcoved triangulation of $r\Delta_{1,d}$. Then $G_\mathcal{T}$ is isomorphic to $G_{r,d}$.
\end{theorem}

\begin{proof}
    We show that the map $\word_1$ is the desired isomorphism from $G_\mathcal{T}$ to $G_{r,d}$. Since it is a bijection between the vertex sets of the graphs, it suffices to check that it preserves edges. We first characterize the edges in the dual graph of the alcoved triangulation. Fix $A_1,A_2\in\mathcal{A}(r\Delta_{1,d})$, and denote the decorated matrix of $\mathcal{I}_{A_i}$ by $M_i$. Note that $A_1$ and $A_2$ form an edge in $G_\mathcal{T}$ if and only if they have $d$ elements in common. In terms of these decorated matrices, $A_1 \sim A_2$ if and only if $M_1$ and $M_2$ have $d$ rows in common. To ease notation, let $w^{(i)} = \word_1(A_i)$ be the words associated to these alcoves, let $R_j$ be the $j$th row of $M_1$ consisting of the elements $R_{j,k}$ for $k=1,2,\ldots,r$. Then, if the mark between $R_j$ and $R_{j+1}$ is in column $m$ $(=w^{(1)}_j)$, the elements of the rows satisfy $R_{j+1,m} = R_{j,m}+1$ and $R_{j+1,\ell} = R_{j,\ell}$ for $\ell\neq m$.
    
    Suppose $A_1 \sim A_2$. Assume that row $i$ of $M_1$ is the row that does not appear in $M_2$. We divide the argument in cases by the value of $i$. 
    
    If $i=1$, we claim that it is not possible to insert a new \emph{different} row in between rows $j$ and $j+1$ of $M_1$ with $j=2,3,\ldots,d$ while maintaining the sorted set conditions. Inserting a new row $R'$ with elements $R'_{s}$ in between $R_j$ and $R_{j+1}$ would imply the following equations in order to be a valid matrix for a sorted set: $$R_{j,\ell} = R'_\ell = R_{j+1,\ell} \;\;\text{ for }\;\; \ell \neq m,\;\; \text{and}$$ $$R_{j,m}+1 = R'_m+1 = R_{j+1,m} \;\;\text{ or }\;\; R_{j,m}+1 = R'_m = R_{j+1,m}.$$ In the first case $R' = R_j$ and in the second $R' = R_{j+1}$. Thus, $M_2$ consists of rows $2,3,\ldots,d+1$ of $M_1$ and a new row appended in the end. 
    In terms of words, this means that $w^{(2)}_t = w^{(1)}_{t+1}$ for $t=1,2,\ldots,d-1$. Note that the mark in row $1$ of $M_1$ is in column $w^{(1)}_1 > 1$. We have the inequality since $w^{(1)}_1 = 1$ would mean that $R_{2,1}=2$ which cannot be the initial value of a sorted set corresponding to an alcove of $r\Delta_{1,d}$ by \cref{lem:properties-marked-matrix-alcoves-dilated-simplex}. Thus, the row $N$ that we may add to get $M_2$ has elements $(N_1,N_2,\ldots,N_r)$ with $N_m = R_{d+1,m}$ for $m\neq w^{(1)}_1-1$ and $N_{w^{(1)}_1-1} = R_{d+1,w^{(1)}_1-1}+1$. That is, $w^{(2)}_d = w^{(1)}_{1}-1$. This is equivalent to condition 1. in \cref{def:G_rd}. The case $i=d+1$ is the counterpart of $i=1$, hence the reasoning is analogous and we obtain condition 1. as it is written in the definition instead. 

    Now suppose $1<i<d+1.$ Let $w^{(1)}_{i-1} = c$ and $w^{(1)}_{i} = c'$. Note that the rows of $M_1$ satisfy
    \begin{alignat*}{3}
        R_{i-1,m} &= R_{i,m} &&= R_{i+1,m} && \quad \text{ for } m\neq c,c' \\
        R_{i-1,c} &= R_{i,c} + 1 &&= R_{i+1,c} + 1 && \\
        R_{i-1,c'} &= R_{i,c'} &&= R_{i+1,c'} + 1 && 
    \end{alignat*}
    Hence rows $R_{i-1}$ and $R_{i+1}$ cannot be adjacent in $M_2$ by the uniqueness of marks in between rows of decorated matrices from \cref{lem:properties-marked-matrix-alcoves-dilated-simplex}. Then $M_2$ consists of rows $R_1,\ldots,R_{i-1},N,R_{i+1},\ldots,R_{d+1}$ where $N$ is the new row added. Such new row satisfies the equations $N_{i-1,m} = N_{i,m} = N_{i+1,m}$ for $m\neq c,c'$ and
    \begin{align*}
        N_{i-1,c} = N_{i,c} + 1 = N_{i+1,c} + 1 &\quad\text{ or }\quad N_{i-1,c} = N_{i,c} = N_{i+1,c} + 1,\\
        N_{i-1,c'} = N_{i,c'} = N_{i+1,c'} + 1 &\quad\text{ or }\quad N_{i-1,c'} = N_{i,c'} + 1 = N_{i+1,c'} + 1 
    \end{align*}
    In order to obtain a valid matrix we need to pick either both conditions on the left tor both conditions on the right; and to avoid the case $N = R_i$, the two conditions on the right are the correct choice. In terms of words we have $w^{(2)}_t = w^{(1)}_t$ for $t=1,\ldots,i-2,i+1,\ldots,d$; $w^{(2)}_{i-1} = w^{(1)}_i$ and $w^{(2)}_i = w^{(1)}_{i-1}.$ This is equivalent to condition 2. in \cref{def:G_rd}. 

    The constructions from before can be done in reverse to see that given two words $w_1,w_2 \in [r]^d$ that form an edge in $G_{r,d}$, the corresponding alcoves $A_1$ and $A_2$ in $\mathcal{A}(r\Delta_{1,r})$ are adjacent in $G_\mathcal{T}.$ 
\end{proof}

\begin{remark}
    The adjacency of alcoves in the standard simplex, and more generally in arbitrary alcoves, was described in \cite[Section 2.5]{lam2007alcoved}. We included the previous proof to make the document self-contained and to verify the connection to the words of the adjacent alcoves. 
\end{remark}

The graph $G_{r,d}$ plays an important role when it comes to dual graphs of alcoved triangulations of polytopes. We will describe one now, and we postpone the second one to the end of the next section. From \cref{thm:alcoves-sorted-sets} it follows that for an alcoved polytope $P$ in $\mathbb{R}^{d+1}$ with $(H,x)$-representation such that it lies in the hyperplane $x_1+\ldots+x_{d+1} = r$, the dual graph of the alcoved triangulation $G_\mathcal{A}$ is a connected subgraph of $G_{r,d}.$ Nevertheless, not every connected subgraph of $G_{r,d}$ can be realized as a dual graph. This can be seen in \cref{fig:alcoved_example} where we can pick four alcoves and such that their union is non-convex and the graph of their adjacency is connected. Hence, the following natural question arises. 

\begin{question}
    What connected subgraphs of $G_{r,d}$ correspond to dual graphs of alcoved triangulations?
\end{question}

We finish this section finding the alcoves in $\mathcal{A}(r\Delta_{1,d})$ whose intersection with the facets of the dilated hypersimplex has codimension $1$. These alcoves are an important in \cref{sec:dual_graph_hypersimplex_general}. We denote that set by $\mathcal{A}^\circ (r\Delta_{1,d})$, and refer to it as the set of \emph{boundary alcoves}. Additionally, we need the following map on words. For $j=1,2,\ldots,d-1$, let $\delta_j$ be the \emph{$j$th duplication map} computed as follows. For a word $v = v_1\,v_2\,\ldots\,v_{d-1}$, suppose $k$ is the index such that $v_k$ is the $j$th letter in the order $1<2<\ldots<d-1$ that is found when reading $v$ from left to right and looping around if needed. Then $\delta_j(v)$ duplicates the letter $v_k$ by adding a copy of it to its right. Note that $\delta_j(v)$ has length $d$. 

\begin{example}
    Let $v = 3\,2\,5\,4\,2\,1\,3$. The third letter that is read in the cyclic order described before is in position $k=5$, so $\delta_3(v) = 3\,2\,5\,4\,2\,\underline{2}\,1\,3$ where the underlined number is the one added by the duplication map.
\end{example}

\begin{prop}\label{prop:boundary_alcoves}
    Let $W^\circ = \{\word_1(A) \,:\, A\in\mathcal{A}^\circ (r\Delta_{1,d})\}$ be the set of words corresponding to the alcoves of $r\Delta_{1,d}$ that intersect the facets of the polytope in codimension $1$. Then $w=w_1\,w_2\,\ldots\,w_d \in W^\circ$ satisfies one of the following conditions:
    \begin{enumerate}
        \item $w_1 = 1$,
        \item $w_{d} = d$, or
        \item $w\in \left\{\delta_j(v) \,:\, v\in [r]^{d-1}\right\}$ for some $j\in\{1,2,\ldots,d-1\}$.
    \end{enumerate}
\end{prop}

\begin{proof}
    Recall that facets of $r\Delta_{1,d}$ are determined by the hyperplanes $H_j = \{\vec{x}\in\mathbb{R}^{d+1}\,:\,x_j = 0\}$ for $j=1,2,\ldots,d+1.$ We consider the case of $H_1$ first. If $A\in\mathcal{A}(r\Delta_{1,d})$ intersects $H_1$ in codimension $1$, this means that $A$ has $d$ vertices in this hyperplane. Hence, $d$ rows of the decorated matrix associated to $A$ have no ones. The extra row of the matrix has a $1$ since $A$ has dimension $d$ and is contained in the dilated hypersimplex. Thus, this row is the first row, and contains exactly one $1$ in the first column. Since the second column has no ones, the first mark of $\mathcal{I}_A$ is in column $1$. This means that the first letter of $\word_1(A)$ is a ``$1$''. Similarly, the same argument can be modified to show that if $A$ intersects $H_{d+1}$ in codimension $1$, then the last letter of $\word_1(A)$ is a ``$d$''. This shows the first two cases. 

    Now consider an alcove $A$ that intersects $H_j$, for $1<j<d+1$, in codimension $1$. Let $M'$ be the decorated matrix corresponding to the rows of $M_{\mathcal{I}_A}$ that do not contain the entry ``$j$''. As before, the extra row has an entry ``$j$'', and then it has to be placed after the $(j-1)$th mark of $M'$. In terms of $\word_1(A)$, this means that the $(j-1)$th letter in the reading order determined by the top-to-bottom and left-to-right order in the matrix is doubled matching the action of $\delta_{j-1}$ on the word of $M'$. Since $M'$ corresponds to an alcove of a facet of $r\Delta_{1,d}$, its reading word is an element of $[r]^{d-1}$, which shows the third case. 
\end{proof}

\begin{example}
    Let $r=3$ and consider the $r$th dilation of the standard simplex of dimension $d=3$. \cref{fig:dual-graph-standard-3-3} shows the graph $G_{3,3}$. The words in different color are the words $w\in[3]^3$ that satisfy $w_1 = 1$, while the words inside of boxes correspond to the image of the duplication map $\delta_1$. Note that the first set corresponds to alcoves intersecting the hyperplane $H_1$ and the second one intersecting $H_2$ (despite the fact that $j=1$ in the duplication map).
    \begin{figure}[t]
        \centering
        \includegraphics[width=0.6\linewidth]{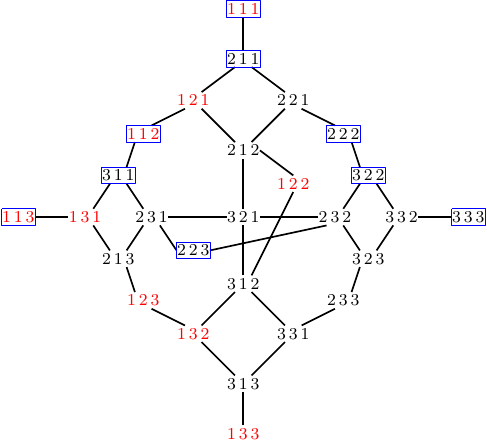}
        \caption{Illustration of $G_{3,3}$, which is isomorphic to the dual graph of the alcoved triangulation of $3\Delta_{1,3}$. The words in different color correspond to alcoves that intersect the hyperplane $x_1 = 0$ in codimension $1$. Similarly, the words in boxes correspond to alcoves intersecting $x_2=0$ in codimension $1$.}
        \label{fig:dual-graph-standard-3-3}
    \end{figure}
\end{example}

\begin{definition}\label{def:w_circ_j}
    To make the notation more explicit, for $j=1,2,\ldots,d+1$, denote by $W^\circ_j$ the set of words in $[r]^d$ that correspond to boundary alcoves of $\mathcal{A}(r\Delta_{1,d})$ with respect to the hyperplane $H_j$. 
\end{definition}

\begin{remark}
    The key idea for the proof of the previous proposition is to consider the intersection of the alcove and the hyperplane that defines a facet as an alcove of a lower dimensional dilated simplex and then adding the correct new row that turns it into a valid alcove for the full dimensional simplex.
\end{remark}

\begin{example}
    Fix $j=3$ and take $w = 2\,4\,5\,6\,1 \in [6]^5$ and suppose $A\in\mathcal{A}(6\Delta_{1,6})$ is such that the reading word of $A\cap H_3$ is $w$. Then the process of the third case of the proof of \cref{prop:boundary_alcoves} can be represented diagrammatically by considering the decorated matrix with word $w$ as being filled skipping $j=3$ and then splitting the $j$th mark to insert the new row. That is,  
    
    \begin{center}
        \begin{tikzpicture}[scale = 0.7]
            \def \w{1};
            \def \h{1};
            \def \r{0.3};
            \def \rows{6};
            \def \cols{6};
            
            \begin{scope}[xshift=0cm,yshift = 0.5cm]
            \foreach \i in {0,...,\rows}
            {
            \draw[gray!50] (0,\i*\h)--(\w*\cols,\i*\h);
            }
            \foreach \i in {0,...,\cols}
            {
            \draw[gray!50] (\w*\i,0)--(\w*\i,\rows*\h);
            }
            \foreach \xx\yy\c in {0/5/1,0/4/1,0/3/1,0/2/1,0/1/1,0/0/2,
                                  1/5/2,1/4/4,1/3/4,1/2/4,1/1/4,1/0/4,
                                  2/5/4,2/4/4,2/3/4,2/2/4,2/1/4,2/0/4,
                                  3/5/4,3/4/4,3/3/5,3/2/5,3/1/5,3/0/5,
                                  4/5/5,4/4/5,4/3/5,4/2/6,4/1/6,4/0/6,
                                  5/5/6,5/4/6,5/3/6,5/2/6,5/1/7,5/0/7
                                  }
                {
                \node at (\w*.5+\w*\xx,\h*.5+\h*\yy) {\c};
                }
                \foreach \xx\yy in {1/5,3/4,4/3,0/1,5/2}
                {
                \draw[red,ultra thick] (\xx,\yy) --  (\xx+1,\yy);
                }
            \end{scope}

            \def \rows{7};
            \def \cols{6};

            \begin{scope}[xshift=12cm]
            \foreach \i in {0,...,\rows}
            {
            \draw[gray!50] (0,\i*\h)--(\w*\cols,\i*\h);
            }
            \foreach \i in {0,...,\cols}
            {
            \draw[gray!50] (\w*\i,0)--(\w*\i,\rows*\h);
            }
            \foreach \xx\yy\c in {0/6/1,0/4/1,0/3/1,0/2/1,0/1/1,0/0/2,
                                  1/6/2,1/4/4,1/3/4,1/2/4,1/1/4,1/0/4,
                                  2/6/4,2/4/4,2/3/4,2/2/4,2/1/4,2/0/4,
                                  3/6/4,3/4/4,3/3/5,3/2/5,3/1/5,3/0/5,
                                  4/6/5,4/4/5,4/3/5,4/2/6,4/1/6,4/0/6,
                                  5/6/6,5/4/6,5/3/6,5/2/6,5/1/7,5/0/7
                                  }
                {
                \node at (\w*.5+\w*\xx,\h*.5+\h*\yy) {\c};
                }

            \foreach \xx\yy\c in {5/5/6,4/5/5,3/5/4,2/5/4,1/5/3,0/5/1}
                {
                \node at (\w*.5+\w*\xx,\h*.5+\h*\yy) {\tcb{\c}};
                }

                \foreach \xx\yy in {3/4,4/3,5/2,0/1}
                {
                \draw[red,ultra thick] (\xx,\yy) --  (\xx+1,\yy);
                }

                \foreach \xx\yy in {1/5,1/6}
                {
                \draw[red,ultra thick,dashed] (\xx,\yy) --  (\xx+1,\yy);
                }
            \end{scope}

        \node at (3,-1.0) {$w = 2\,4\,5\,6\,1$};
        \node at (15,-1.0) {$\delta_2(w) = 2\,\underline{2}\,4\,5\,6\,1$};
        \draw[->] (7,3.5)--(11,3.5);
        \draw[->] (7,-1.0)--(11,-1.0);
        \end{tikzpicture}
    \end{center}
\end{example}

\subsection{The dilated hypersimplex}\label{sec:delta-i-d-simplices-labels}

We start by giving a proof of the fact that Eulerian numbers coincide with the normalized volume of hypersimplices by giving a explicit bijection from the set of alcoves to the permutations with fixed number of descents. In the following theorem, the map $\sigma^{(i)}_\bullet$ is defined for alcoves of $\Delta_{i,d}$ in an analogous way as in \cref{def:permutation-alcove-dilated-simplex}.

\begin{theorem}\label{thm:alcoves-hypersimplex-eulerian}
    For $1\leq i \leq d$, the map $$\sigma^{(i)}_{\bullet}\;:\; \mathcal{A}(\Delta_{i,d}) \to \mathfrak{A}(d,i)$$ is a bijection.
\end{theorem}

\begin{proof}
    Let $A\in\mathcal{A}(\Delta_{i,d})$ and $\mathcal{I} = \mathcal{I}_A$ its associated sorted set. We describe the structure of $\widetilde M_{\mathcal{I}}$. First, the elements of $\mathcal{I}$ are $i$-subsets of $\{1,2,\ldots,d+1\}$. This can be see from the $(H,x)$-representation of $\Delta_{i,d}$ in \cref{eq:hypersimplex}. Hence, each row of $\widetilde M_{\mathcal{I}}$ has no repeated entries. Now, following the same arguments for \cref{lem:properties-marked-matrix-alcoves-dilated-simplex} we have that $\widetilde M_{\mathcal{I}}$ has exactly $d$ marks, each in between a pair of adjacent rows. Note that if the $i$-th mark is higher than the $(i+1)$-th mark and in a different column, there are rows with repeated values. This is also the case if there is a column without a mark. Therefore, given a mark in $\widetilde M_{\mathcal{I}}$, the next mark is either lower in the same column or higher in a different column and every column of the matrix has at least one mark. This implies that the permutation $\sigma^{(i)}_A$, obtained by reading the row labels of the marks from top to bottom and left to right, is a permutation of $\{1,2\ldots,d\}$ with exactly $i-1$ descents, i.e. $\sigma^{(i)}_A\in\mathfrak{A}(d,i).$ Moreover, from a permutation $\tau\in\mathfrak{A}(d,i)$, the decorated matrix constructed by placing the marks in the rows given by the one-line notation of $\tau$ and changing columns every times we reach a descent corresponds, after filling it as in the proof of \cref{thm:bijection-alcoves-words}, to an alcove of $\Delta_{i,d}$.
\end{proof}

\begin{example}\label{ex:alcove_hypersimplex}
    Let $i=3$ and $d=6$. The set of points 
    \begin{multline*}
        A = \{(1,0,1,0,1,0,0),(0,1,1,0,1,0,0),(0,1,0,1,1,0,0),(0,1,0,1,0,1,0),\\(0,0,1,1,0,1,0),(0,0,1,1,0,0,1),(0,0,1,0,1,0,1)\}\subseteq \mathbb{R}^7
    \end{multline*}
    defines an alcove of $\Delta_{3,6}\subseteq \mathbb{R}^7_x$. The decorated matrix in this case is 
    \begin{center}
        \begin{tikzpicture}[scale = 0.7]
            \def \w{1};
            \def \h{1};
            \def \r{0.3};
            \def \rows{7};
            \def \cols{3};
            \foreach \i in {0,...,\rows}
            {
            \draw[gray!50] (0,\i*\h)--(\w*\cols,\i*\h);
            }
            \foreach \i in {0,...,\cols}
            {
            \draw[gray!50] (\w*\i,0)--(\w*\i,\rows*\h);
            }
            \foreach \xx\yy\c in {0/6/1,0/5/2,0/4/2,0/3/2,0/2/3,0/1/3,0/0/3,
                                  1/6/3,1/5/3,1/4/4,1/3/4,1/2/4,1/1/4,1/0/5,
                                  2/6/5,2/5/5,2/4/5,2/3/6,2/2/6,2/1/7,2/0/7
                                  }
            {
            \node at (\w*.5+\w*\xx,\h*.5+\h*\yy) {\c};
            }
            \foreach \xx\yy in {0/6,0/3,1/5,1/1,2/4,2,2}
            {
            \draw[red,ultra thick] (\xx,\yy) --  (\xx+1,\yy);
            }
        \end{tikzpicture}
    \end{center}
    and the permutation is $\sigma^{(3)}_A = 1\,4\,2\,6\,3\,5\,\in\mathfrak{A}(6,3).$
\end{example}

With this interpretation of the alcoves of $\Delta_{i,d}$ we can prove that the map $\pair_1$ from the previous section is a bijection. 

\begin{proof}[Proof of \cref{thm:bijection-alcoves-pairs}]

Let $A^{(j)}_\bullet\,:\,\mathfrak{A}(d,j) \to \mathcal{A}(\Delta_{j,d}) $ be the inverse map of $\sigma^{(j)}_\bullet.$ Define the map $$\alc_1\;:\; \bigcup_{j=1}^d \mathfrak{C}(r-1,d+1,r-j)\times \mathfrak{A}(d,j) \to \mathcal{A}(r\Delta_{1,d})$$ by $\alc_1(\vec{c},\sigma) = \vec{c} + A_\sigma^{(j)}$ for $(\vec{c},\sigma) \in \mathfrak{C}(r-1,d+1,r-j)\times \mathfrak{A}(d,j)$. Here the $+$ sign denotes the translation of the set in the direction given by the vector. First, we check that $A':=\alc_1(\vec{c},\sigma) \in \mathcal{A}(r\Delta_{1,d})$. Note that adding $\vec{c}$ to each of the vertices of $A_\sigma^{(j)}$ is reflected in the decorated matrix as adding $(d+1)c_j$ entries ``$j$''  after the $(j-1)$th mark (here we assume the $0$th mark is in the top-left corner of the matrix). Thus, $A'$ is an alcove since $\mathcal{I}_{A'}$ is sorted. Moreover, in the decorated matrix of ${\mathcal{I}_{A'}}$ all entries are at most $d+1$, there are $d+1$ rows and $j+\sum_k \vec{c}_k = r$ columns. Hence, $A'$ is an alcove of $r\Delta_{1,d}$ as desired. Finally, by \cref{lem:comp-A-from-decorated-matrix}, \cref{def:permutation-alcove-dilated-simplex} and the previous description of the map, it follows that $\pair_1$ and $\alc_1$ are inverses. 
\end{proof}

\begin{example}
    Let $\sigma = 1\,4\,2\,6\,3\,5\,\in\mathfrak{A}(6,3)$ as in \cref{ex:alcove_hypersimplex} and $\vec{c} = (1,0,2,0,0,1,0) \in \mathfrak{C}(6,7,4)$ (that is, $r=7$, $d=6$ and $j = 3$ in \cref{thm:bijection-alcoves-pairs}.) The decorated matrix associated to $A' = \alc_1(\vec{c},\sigma)$ is 
    \begin{center}
        \begin{tikzpicture}[scale = 0.7]
            \def \w{1};
            \def \h{1};
            \def \r{0.3};
            \def \rows{7};
            \def \cols{7};
            \foreach \i in {0,...,\rows}
            {
            \draw[gray!50] (0,\i*\h)--(\w*\cols,\i*\h);
            }
            \foreach \i in {0,...,\cols}
            {
            \draw[gray!50] (\w*\i,0)--(\w*\i,\rows*\h);
            }
            \foreach \xx\yy\c in {1/6/1,1/5/2,1/4/2,1/3/2,3/2/3,3/1/3,3/0/3,
                                  4/6/3,4/5/3,4/4/4,4/3/4,4/2/4,4/1/4,4/0/5,
                                  5/6/5,5/5/5,5/4/5,6/3/6,6/2/6,6/1/7,6/0/7
                                  }
            {
            \node at (\w*.5+\w*\xx,\h*.5+\h*\yy) {\c};
            }

            \foreach \xx\yy\c in {0/6/1,0/5/1,0/4/1,0/3/1,0/2/1,0/1/1,0/0/1,
                                  2/6/3,2/5/3,2/4/3,2/3/3,2/2/3,2/1/3,2/0/3,
                                  3/6/3,3/5/3,3/4/3,3/3/3,1/2/3,1/1/3,1/0/3,
                                  6/6/6,6/5/6,6/4/6,5/3/6,5/2/6,5/1/6,5/0/6
                                  }
            {
            \node at (\w*.5+\w*\xx,\h*.5+\h*\yy) {\tcb{\c}};
            }
            \foreach \xx\yy in {1/6,1/3,4/5,4/1,5/4,6/2}
            {
            \draw[red,ultra thick] (\xx,\yy) --  (\xx+1,\yy);
            }
            \foreach \xx\yy in {0/7}
            {
            \draw[red,ultra thick,dashed] (\xx,\yy) --  (\xx+1,\yy);
            }
        \end{tikzpicture}
    \end{center}
    The dashed mark is the $0$th mark mentioned in the theorem, and the entries in a different color are the entries added to the decorated matrix $A^{(3)}_\sigma$ to obtain $A'.$
\end{example}

We now turn our attention to dilated hypersimplices. The $(H,x)$-representation of these polytopes is given by $$r\Delta_{i,d} = \left\{ \vec{x}\in\mathbb{R}_x^{d+1} \;:\; 0\leq x_1\,,\,x_2\,,\,\ldots\,,\,x_{d+1} \leq r \quad \text{and} \quad x_1+x_2+\cdots+x_{d+1} = ir \right\}.$$ The proof of \cref{eq:identity-general-i} follows from the description of a suitable pair of labelings of $\mathcal{A}(r\Delta_{i,d}).$

\subsubsection{Labeling of the alcoves with pairs of words and permutations}

We point out how \cref{thm:bijection-alcoves-words} can be used to understand triangulations of dilated polytopes.

\begin{observation}\label{obs:main_idea}
    Let $P$ be a polytope of dimension $d$ that has a unimodular triangulation $\mathcal{T} = \{S_i \;|\; i\in I\}$. Consider the (non-unimodular) triangulation $r\mathcal{T} = \{rS \,|\, S\in\mathcal{T}\}$ of $rP$. Then we can use the map $\word_1$ from \cref{sec:delta-1-d-simplices-labels} to construct a labeling $f\,:\, \Delta \to I\times [r]^d$ of the unimodular triangulation $\Delta$ that arises from alcove-triangulating each simplex $rS \simeq r\Delta_{1,d}$.
\end{observation}

\begin{remark}
    These triangulations were considered in \cite[Section 4]{Haase_2021} where the authors study the properties of the resulting triangulations and related questions with unimodular triangulations of dilated polytopes. 
\end{remark}

Using this idea we can give a labeling of the alcoves of a dilated hypersimplex as follows. Consider $A\in\mathcal{A}(r\Delta_{i,d})$. It satisfies $A \subseteq rB$ where $B\in\mathcal{A}(\Delta_{i,d})$. Define the \emph{permutation associated to $A$} to be $\tau_A = \sigma_B^{(i)}$. Moreover, there is an affine equivalence $\varphi_B:B\to\Delta_{1,d}$. Through this map, $\varphi_B(A)$ is an alcove of $r\Delta_{1,d}$, and we can compute its $\word_1$. Define the \emph{word of $A$} to be $\word_i'(A) = \word_1(\varphi_B(A))$. Using these objects together with \cref{thm:alcoves-hypersimplex-eulerian}, we obtain the first labeling of $\mathcal{A}(r\Delta_{i,d})$. 

\begin{theorem}\label{thm:bijection-alcoves-words-general-case}
    The map $$\words_i \; : \; \mathcal{A}(r\Delta_{i,d}) \to [r]^d\times \mathfrak{A}(d,i)$$ defined by $\words_i(A) = (\word_i'(A),\tau_A)$ is a bijection.  
\end{theorem}

\subsubsection{Labeling of the alcoves with pairs of compositions and permutations}

Now we describe a more direct labeling of $\mathcal{A}(r\Delta_{i,d})$ by extending the maps from \cref{sec:labeling_pair_case_1}. For alcoves $A\in\mathcal{A}(r\Delta_{i,d})$, although the structure of the decorated matrix $\widetilde M_{\mathcal{I}_A}$ is different from the $i=1$ case, we can still construct a composition and a permutation in a similar way as we did for $r\Delta_{1,d}$. In fact, \cref{lem:properties-marked-matrix-alcoves-dilated-simplex} is also valid for alcoves of $\mathcal{A}(r\Delta_{i,d})$, so \cref{lem:comp-A-from-decorated-matrix}, which now yields a composition of $ir-j$ for some $0\leq j\leq d$, is still valid and \cref{def:permutation-alcove-dilated-simplex} can be mimicked in to define a permutation of the alcoves; moreover, the proof of \cref{prop:properties-comp-A-and-permutation-A} is still valid. Hence, we can use these to obtain the desired labeling. The proof that this is indeed a labeling is analogous to the proof of \cref{thm:bijection-alcoves-pairs}. For $A \in \mathcal{A}(r\Delta_{i,d})$ denote by $\comp'(A)$ and $\sigma_A'$ the composition and permutation (respectively) obtained from the decorated matrix $\widetilde M_{\mathcal{I}_A}$.

\begin{theorem}\label{thm:bijection-alcoves-pairs-general-case}
    The map $$\pair_i\,:\,\mathcal{A}(r\Delta_{i,d})\to \bigcup_{j=1}^d \mathfrak{C}(r-1,d+1,ir-j)\times \mathfrak{A}(d,j)$$ given by $\pair_i(A) = (\comp'(A),\sigma'_A)$ is a bijection.
\end{theorem}

\noindent From \cref{thm:bijection-alcoves-words-general-case,thm:bijection-alcoves-pairs-general-case}, we obtain a combinatorial proof of \cref{eq:identity-general-i}. 

\begin{example}
    Fix $d=5$, $r=4$, $i=2$ and $j=3$ as parameters in the previous theorems. The set of points $$A = \{(2,3,0,1,2,0),(2,2,1,1,2,0),(2,2,0,2,2,0),(2,2,0,2,1,1),(1,3,0,2,1,1),(1,3,0,1,2,1)\}$$ defines an alcove in $\mathcal{A}(4\Delta_{2,5})$. It satisfies $\conv(A) \subseteq 4A^{(2)}_{3\,1\,2\,4\,5}$, so $\tau_A = 3\,1\,2\,4\,5$. Moreover, $A^{(2)}_{3\,1\,2\,4\,5} = \conv(\mathfrak{B})$ with 
    \begin{align*}
        \mathfrak{B} &= \{ \vec{v}_1,\vec{v}_2,\vec{v}_3,\vec{v}_4,\vec{v}_5,\vec{v}_6\} \\
        &= \left\{ (1,1,0,0,0,0),(1,0,1,0,0,0),(1,0,0,1,0,0),(0,1,0,1,0,0),(0,1,0,0,1,0),(0,1,0,0,0,1)\right\}.
    \end{align*}
    Using $\mathfrak{B}$ as a basis for $\mathbb{R}_x^6$, the elements of $A$ can be rewritten as $$A \cong \{(1,0,1,0,2,0),(1,1,0,1,1,0),(0,0,2,0,2,0),(0,0,2,0,1,1),(0,0,1,1,1,1),(0,0,1,0,2,1)\}$$ where $\cong$ denotes the change of basis from the standard basis to $\mathfrak{B}$. From this description we see that the decorated matrix of $A$ relative to $4A^{(2)}_{3\,1\,2\,4\,5}$ is 
    \begin{center}
        \begin{tikzpicture}[scale = 0.7]
            \def \w{1};
            \def \h{1};
            \def \r{0.3};
            \def \rows{6};
            \def \cols{4};
            \foreach \i in {0,...,\rows}
            {
            \draw[gray!50] (0,\i*\h)--(\w*\cols,\i*\h);
            }
            \foreach \i in {0,...,\cols}
            {
            \draw[gray!50] (\w*\i,0)--(\w*\i,\rows*\h);
            }
            \foreach \xx\yy\c in {0/5/1,0/4/2,0/3/3,0/2/3,0/1/3,0/0/3,
                                  1/5/3,1/4/3,1/3/3,1/2/3,1/1/4,1/0/5,
                                  2/5/5,2/4/5,2/3/5,2/2/5,2/1/5,2/0/5,
                                  3/5/5,3/4/5,3/3/5,3/2/6,3/1/6,3/0/6
                                  }
            {
            \node at (\w*.5+\w*\xx,\h*.5+\h*\yy) {\c};
            }
            \foreach \xx\yy in {0/5,0/4,1/2,1/1,3/3}
            {
            \draw[red,ultra thick] (\xx,\yy) --  (\xx+1,\yy);
            }
        \end{tikzpicture}
    \end{center}
    Putting all the information together we obtain $\words_i(A) = (1\,1\,4\,2\,2 \;,\; 3\,1\,2\,4\,5) \in [4]^5 \times \mathfrak{A}(5,2).$ 
    
    To compute $\pair_i(A)$ we consider the decorated matrix of $\mathcal{I}_A$ with respect to the canonical basis of $\mathbb{R}_x^6$. That is,
    \begin{center}
        \begin{tikzpicture}[scale=0.7]
                \def \w{1};
                \def \h{1};
                \def \r{0.3};
                \def \rows{6};
                \def \cols{8};
                \foreach \i in {0,...,\rows}
                {
                \draw[gray!50] (0,\i*\h)--(\w*\cols,\i*\h);
                }
                \foreach \i in {0,...,\cols}
                {
                \draw[gray!50] (\w*\i,0)--(\w*\i,\rows*\h);
                }
                \foreach \xx\yy\c in {0/5/1,0/4/1,0/3/1,0/2/1,0/1/1,0/0/1,
                                      1/5/1,1/4/1,1/3/1,1/2/1,1/1/2,1/0/2,
                                      2/5/2,2/4/2,2/3/2,2/2/2,2/1/2,2/0/2,
                                      3/5/2,3/4/2,3/3/2,3/2/2,3/1/2,3/0/2,
                                      4/5/2,4/4/3,4/3/4,4/2/4,4/1/4,4/0/4,
                                      5/5/4,5/4/4,5/3/4,5/2/4,5/1/4,5/0/5,
                                      6/5/5,6/4/5,6/3/5,6/2/5,6/1/5,6/0/5,
                                      7/5/5,7/4/5,7/3/5,7/2/6,7/1/6,7/0/6
                                      }
                    {
                    \node at (\w*.5+\w*\xx,\h*.5+\h*\yy) {\c};
                    }
                    \foreach \xx\yy in {1/2,4/4,4/5,5/1,7/3}
                    {
                    \draw[red,ultra thick] (\xx,\yy) --  (\xx+1,\yy);
                    }
        \end{tikzpicture}
    \end{center}
    and from this matrix we obtain $\comp'(A) = (1,2,0,1,1,0) \in \mathfrak{C}(3,6,5)$ and $\sigma'_A = 4\,1\,2\,5\,3 \in \mathfrak{A}(5,3).$ 
\end{example}

\subsubsection{Dual graph of the triangulation}\label{sec:dual_graph_hypersimplex_general}

Let $G_{\mathcal{A},i,d}$ be the dual graph of the alcoved triangulation of $\Delta_{i,d}$. We present a description of this graph using permutations in view of \cref{thm:alcoves-hypersimplex-eulerian}. The proof of this characterization is analogous to the one from \cref{thm:dual_graph_alcoved_dilated_standard_simplex}. 

\begin{prop}\label{prop:dual_graph_hypersimplex}
    The graph $G_{\mathcal{A},i,d}$ is isomorphic to the graph with vertex set $\mathfrak{A}(d,i)$ where $\{\sigma,\tau\}$ is an edge if and only if the permutations satisfy that
    \begin{enumerate}
        \item $\sigma = s_k \tau$ for some $k=1,2,\dots,d-1$ (where $s_k$ is the $k$th simple transposition), or
        \item if $\sigma = \sigma_1\,\sigma_2\,\ldots\,\sigma_d$ and $\tau = \tau_1\,\tau_2\,\ldots\,\tau_d$ are the one-line notations, and $\ell$ and $m$ are the indices such that $\sigma_\ell = 1$ and $\sigma_m=d$, then either
        \begin{enumerate}
            \item $\tau_\ell = d$ and $\tau_j = \sigma_j+1$ for $j\neq\ell$, or
            \item $\tau_m = 1$ and $\tau_j = \sigma_j-1$ for $j\neq m$.
        \end{enumerate}
    \end{enumerate}
\end{prop}

\begin{example}
    \cref{fig:dual_graph_delta-2-4_ss-perms} shows the alcoved triangulation of $\Delta_{2,4}$ in terms of decorated matrices of sorted sets and also in terms of permutations from $\mathfrak{A}(4,2)$.
    \begin{figure}[t]
        \centering
        \hspace{0.75cm}\includegraphics[width=0.9\linewidth]{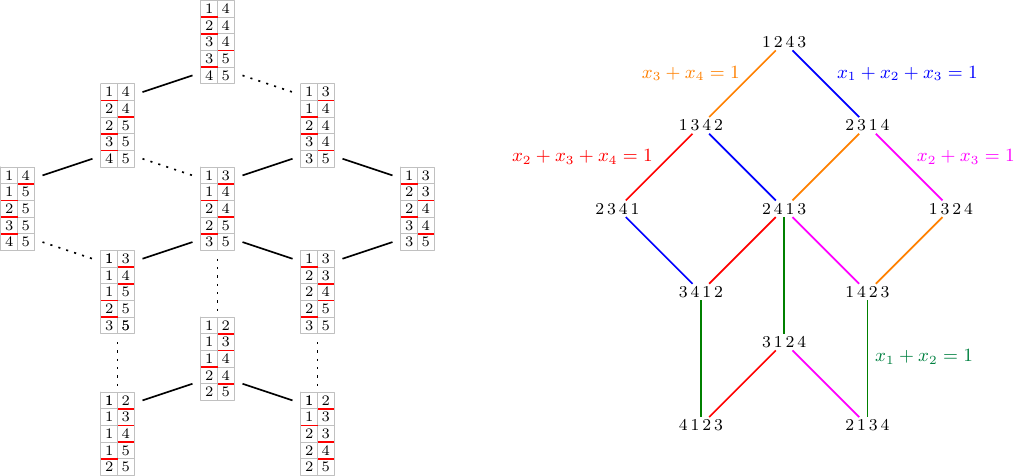}
        \hspace{1cm}
        \caption{The dual graph of the alcoved triangulation of $\Delta_{2,4}$ in terms of sorted sets on the left; the solid and dotted edges correspond, respectively, to conditions $1.$ and $2.$ from \cref{prop:dual_graph_hypersimplex}. On the right, the same graph but in terms of permutations in $\mathfrak{A}(4,2)$; the colors of the edges represent the different hyperplanes that separate the corresponding alcoves.}
        \label{fig:dual_graph_delta-2-4_ss-perms}
    \end{figure}
    In the graph with decorated matrices, the solid edges correspond to changing a middle row (i.e. not the first or last) of the matrix, and the dotted edges change a top or bottom row of the matrix. These are the analogous cases from \cref{thm:dual_graph_alcoved_dilated_standard_simplex} and correspond to conditions 1. and 2. from \cref{prop:dual_graph_hypersimplex} respectively. In the graph with permutations we show the different hyperplanes that generate the triangulation with different colors. 
\end{example}

We would like to describe the dual graph of the alcoved triangulation of $r\Delta_{i,d}$ in the setting of \cref{obs:main_idea}, that is, as a \emph{composition of dual graphs}. This is equivalent to describing the graph using the labeling of $\mathcal{A}(r\Delta_{i,d})$ from \cref{thm:bijection-alcoves-words-general-case}. We start by defining the operation on graphs that formalizes this idea.

\begin{definition}\label{def:composition_of_graphs}
    Let $G=(V(G),E(G))$ and $H$ be finite nonempty graphs. For each vertex $v\in V(G)$ take a copy of $H$ and denote it by $H_v$, and for each edge $e=\{x,y\} \in E(G)$ pick a bijection $f_e:X_e\to Y_e$ for some $X_e\subseteq V(H_x)$ and $Y_e\subseteq V(H_y)$. Define $G\langle H \rangle$ as the graph with vertices $$V\left ( G\langle H \rangle \right ) = \bigsqcup_{v\in V(G)} V(H_v)$$ and edges $$E\left ( G\langle H \rangle \right ) = \bigsqcup_{v\in V(G)} E(H_v) \sqcup \bigsqcup_{e\in E(G)}\left\{ \{x,f_e(x)\} \,:\, x\in X_e \right\}$$ where $\sqcup$ denotes disjoint union.   
\end{definition}

Intuitively, the previous construction is as follows. We place a copy of $H$ in each vertex of $G$. Then for each one of the edges $e$ of $G$, we connect the copies corresponding to the endpoints of $e$ by creating edges between the sets $X_e$ and $Y_e$ according to the bijection $f_e$. Because of this, we refer to the sets $X_e$ and $Y_e$ as the \emph{connecting sets}. This notion generalizes the Cartesian product of graphs, which is recovered by setting all the bijections $f_e$ to be the identity on $V(H)$.

For our purposes, we take $G=G_{\mathcal{A},i,d}$ and $H = G_{r,d}$; for the connecting sets we take the boundary words $W_j^\circ$ from \cref{def:w_circ_j} and the bijections are the identity maps between them. However, additional conditions on how to pick the order of connection of the boundaries must be taken into consideration, as the next example shows.

\begin{example}\label{ex:graph-composition}
Consider the case of $2\Delta_{2,3}$. Then the relevant graphs for the construction are
\begin{figure}[ht]
    \centering
    \includegraphics[width=0.6\linewidth]{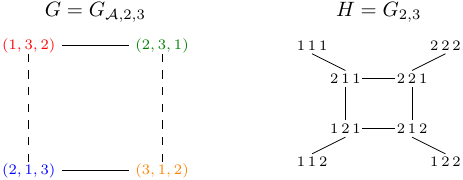}
\end{figure}    

\noindent In $G_{\mathcal{A},2,3}$ we wrote the one-line notation of the permutations with brackets and commas to avoid confusion to the words in $[2]^3$. As in \cref{fig:dual_graph_delta-2-4_ss-perms}, the solid edges of $G$ are encoding adjacency of the simplices through the hyperplane $x_2+x_3 = 1$ and the dashed edges through the hyperplane $x_1+x_2 = 1$. Since there are two pairs of edges encoding the same adjacency, the connecting sets from $H_{(1,3,2)}$ to $H_{(2,3,1)}$, and from $H_{(2,1,3)}$ to $H_{(3,1,2)}$ have to coincide. The same happens for the dashed edges. Thus, picking the set $W_1^\circ \subset [2]^3$ for the solid edges and $W_2^\circ\subset [2]^3$ for the dashed edges, we obtain the connections that are shown in \cref{fig:compo-dual-graphs}. Picking different connecting sets for edges of the same type yields a graph that does not correspond to the alcoved triangulation of the dilated polytope. This can be seen from the existence of the $4$-cycle including the four different instances of the word $1\,1\,1$, and the fact that picking the sets incorrectly does not produce such cycle. 

\begin{figure}[ht]
    \centering
    \includegraphics[width=0.7\linewidth]{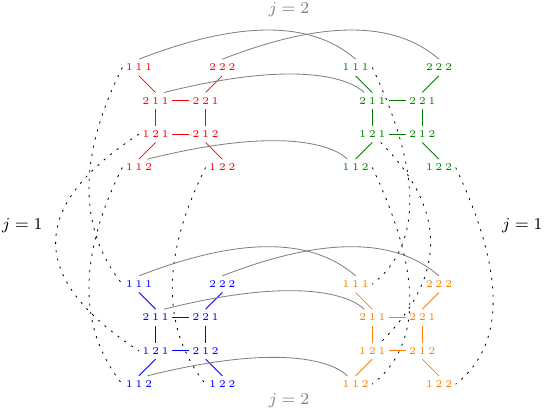}
    \caption{The construction of $G_{\mathcal{A},2,3}\langle G_{2,3} \rangle$ when choosing $W_j^\circ$ as connecting sets according to the specified labels. This graph is isomorphic to the dual graph of the alcoved triangulation of $2\Delta_{2,3}$.}
    \label{fig:compo-dual-graphs}
\end{figure}

\end{example}

Since there is extra compatibility needed depending on the type of hyperplanes to which the edges of the graph $G = G_{\mathcal{A},i,d}$ correspond, we need to determine the correct connecting sets. From the reasoning in \cref{ex:graph-composition}, we propose the following candidates. Every edge of $G$ corresponds to a hyperplane from the affine Coxeter arrangement of type $A_{d+1}$. Thus, by identifying the edges with the hyperplanes, we obtain an edge-coloring of $G$ (see \cref{fig:dual_graph_delta-2-4_ss-perms}) that, we believe, induces a correct choice of connecting sets. 

\begin{conjecture}\label{conj:dual_graph_alcoved_dilated_hypersimplex}
    Let $G = G_{\mathcal{A},i,d}$ and $H = G_{r,d}$. The edge-coloring of $G$ determined by the hyperplane types prescribes a choice of connecting sets that make $G\langle H \rangle$ isomorphic to the dual graph of the alcoved triangulation of $r\Delta_{i,d}$. 
\end{conjecture}

\section*{Acknowledgements}

This work was partially supported by NSERC grant RGPIN-2021-02568. The author is grateful to Volkmar Welker for pointing out that there was not known combinatorial proof of the identities. Also to Santiago Estupi\~n\'an Salamanca, Cicely Henderson, and Federico Castillo for valuable conversations and comments on the document, and to Sophie Spirkl for pointing out the connection of the graph operation with Cartesian products.

\bibliographystyle{abbrv}
\bibliography{references}

\end{document}